\documentclass{aims}
\usepackage{amsmath,amssymb,color,url}
  \usepackage{paralist}
  \usepackage{graphics} 
  \usepackage{epsfig} 
\usepackage{graphicx}  \usepackage{epstopdf}
 \usepackage[colorlinks=true]{hyperref}
\hypersetup{urlcolor=blue, citecolor=red}

\def\currentvolume{X}
 \def\currentissue{X}
  \def\currentyear{200X}
   \def\currentmonth{XX}
    \def\ppages{X--XX}
     \def\DOI{10.3934/xx.xx.xx.xx}

\newtheorem{theorem}{Theorem}[section]
\newtheorem{corollary}{Corollary}

\newtheorem{lemma}[theorem]{Lemma}
\newtheorem{proposition}{Proposition}

\theoremstyle{definition}

\newtheorem{remark}{Remark}

\newcommand\Ltwow{{L^2_{c^2/c_0^2}(\Omega)}}
\newcommand\Ltwo{{L^2(\Omega)}}

\newcommand\kappaintro{\kappa}
\newcommand\cutoff{\chi}
\newcommand\dkappa{\underline{d\kappa}}

\newcommand\db{\underline{db}}
\newcommand\dc{\underline{dc}}

\newcommand\Lapl[1]{\widehat{#1}}

\newcommand\acoeff{\mathfrak{a}}
\newcommand\rcoeff{\mathfrak{r}}
\newcommand\obsop{\mathcal{C}_\Sigma}

\newcommand\calAc{\mathcal{A}_c}

\newcommand{\nLtwo}[1]{\|#1\|_{L^2}}
\newcommand{\nLtwoLtwo}[1]{\|#1\|_{L^2 (L^2)}}
\newcommand{\normE}[1]{\|#1\|_{\mathcal{E}}}
\newcommand{\Honetwo}{H^2_\diamondsuit(\Omega)}

\newcommand\utau{u^\tau}
\newcommand\bconst{\bar{b}}

\title[Uniqueness of space dependent coefficients in the JMGT equation] 
      {Uniqueness of some space dependent coefficients in a wave equation of nonlinear acoustics}

\author[Barbara Kaltenbacher]{}

\subjclass{Primary: 35R30; Secondary: 35B30, 35Q99. }

 \keywords{nonlinearity parameter tomography, JMGT equation, nonlinear acoustics.}

 \email{barbara.kaltenbacher@aau.at}

\thanks{This work was supported by the Austrian Science Fund FWF  http://dx.doi.org/10.13039/501100002428 under the grant P36318}

\begin{document}
\maketitle

\centerline{\scshape Barbara Kaltenbacher}
\medskip
{\footnotesize
 \centerline{Department of Mathematics}
   \centerline{Alpen-Adria-Universit\"at Klagenfurt}
   \centerline{ Universit\"atsstra\ss e 65-67, 9020 Klagenfurt, Austria}
} 

\bigskip

 \centerline{(Communicated by the associate editor name)}

\begin{abstract}
In this paper we prove uniqueness for some parameter identification problems for the JMGT equation, a third order in time quasilinear PDE in nonlinear acoustics.
The coefficients to be recovered are the space dependent nonlinearity parameter, sound speed, and attenuation parameter, and the observation available is a single time trace of the acoustic pressure on the boundary.
This is a setting relevant to several ultrasound based tomography methods.
Our approach relies on the Inverse Function Theorem, which requires to prove that the forward operator is a differentiable isomorphism in appropriately chosen topologies and with an appropriate choice of the excitation. 
\end{abstract}

\section{Introduction}
The consideration and even exploitation of nonlinearity in ultrasound imaging has found much interest in the engineering and medical literature and is recently also starting to become a topic of mathematical research in inverse problems. 

We consider, as one of the advanced models of nonlinear acoustics, the Jordan-Moore-Gibson-Thompson (JMGT) equation \cite{Jordan2014,MooreGibson1960,Thompson1972}
\begin{equation}\label{JMGT}
\tau u_{ttt} + u_{tt}-c_0^2\Delta u - \tau c_0^2\Delta u  + D [u] = \kappaintro(u^2)_{tt} + r,
\end{equation}
where $u$ is the acoustic pressure, $c_0$ the speed of sound, 
$\tau$ the relaxation time, and $\kappaintro$ 
$=\frac{\beta_a}{\varrho c_0^2}=\tfrac{1}{\varrho c_0^2}(\tfrac{B}{2A}+1)$
contains the nonlinearity parameter $\beta_a$ or $B/A$ along with the sound speed $c_0$ and the mass density $\varrho$.
The JMGT equation can be written equivalently as a second order wave type euation in terms of an auxiliary quantity $\utau$
\begin{equation}\label{JMGTz}
\begin{aligned}
&\utau_{tt}-c_0^2\Delta \utau  + \tilde{D} [\utau,u] = \kappaintro(u^2)_{tt} + r,\\ 
&\utau=\tau u_t+u
\end{aligned}
\end{equation}
The choice of the damping term $D[u]$ or $\tilde{D}[\utau,u]$ is highly relevant for the degree of ill-posedness of inverse problems related to the PDE \eqref{JMGT}.
In view of the fact that strong damping makes the forward problem behave like a parabolic PDE, which renders the inverse problem severely ill-posed, we use weak damping only. A term of the classical weak damping form $D[u]=b_0u_t$ would not suffice to yield decay of the wave energy, though. We therefore employ weak damping in terms of the auxiliary quantity $\utau$
\begin{equation}\label{D}
\tilde{D}[\utau] = b_0 \utau_t, \text{ that is, } D[u] = b_0(\tau u_t+u)_t 
\end{equation}
with $b_0\geq0$ in this paper, while alternative (e.g., fractional) damping terms 
\cite{nonlinearity_imaging_fracWest} as relevant in ultrasonics might be studied in follow-up work.

This PDE contains several coefficients that are specific to the type of tissue. They thus have to be regarded as function of space and therefore provide a means of imaging.
For example, in ultrasound tomography 
\cite{AlsakerCardenasFuruieMueller2021,Gemmeke_etal:2017,Greenleaf_etal:1974,JavaherianLuckaCox20,LuckaPerezlivaTreebyCox2022,MuellerCardenasFuruie2021}, 
the sound speed $c_0=c_0(x)$ is the imaged quantity; in nonlinearity parameter tomography, it is $\kappaintro=\kappaintro(x)$
\cite{Bjorno1986, BurovGurinovichRudenkoTagunov1994, Cain1986, 
IchidaSatoLinzer1983, PanfilovaSlounWijkstraSapozhnikovMischi:2021, VarrayBassetTortoliCachard2011,ZhangChenGong2001, ZhangChenYe1996}.
Also the attenuation coefficient $b_0=b_0(x)$ is known to contain tissue specific information 
\cite{DineKak:1979,LiDuricHuang:2008,Perez-LivaCoxTreeby:2017}.
Mathematically, recovery of these coefficients from the typical observations available in this context, namely measurements of the acoustic pressure at some receiver array outside the body, leads to coefficient identification problems in a PDE of the type \eqref{JMGT} from boundary observations of the acoustic pressure $u$.

The pressure data $h$ taken at the receivers can be written as a Dirichlet trace on some manifold $\Sigma$ immersed in the domain $\Omega$ or attached to its boundary $\Sigma\in\overline{\Omega}$
\begin{equation}\label{observation}
h(t,x) = u(t,x), \quad(t,x)\in(0,T)\times \Sigma.
\end{equation}
$\Sigma$ models the transducer or hydrophone array and may as well just be a subset of discrete points on a manifold.

We assume \eqref{JMGT} to hold in a smooth and bounded domain 
$\Omega\subseteq\mathbb{R}^d$, $d\in\{1,2,3\}$ and equip it with initial conditions $u(t=0)=u_0$, $u_t(t=0)=u_1$, $u_{tt}(t=0)=u_2$, as well as homogeneous impedance boundary conditions 
\begin{equation}\label{impedance}
\gamma_0 u+\gamma_1\partial_\nu u = 0 \mbox{ on }(0,T)\times\partial\Omega
\end{equation}
with $\gamma_0,\gamma_1\in L^\infty(\partial\Omega)$, $\gamma_0,\gamma_1\geq0$, $\gamma_0+\gamma_1\geq\underline{\gamma}>0$; the case of vanishing $\gamma_1$ or $\gamma_0$ represents Dirichlet or Neumann conditions, respectively.
The space- and time-dependent 
source term $r$ in \eqref{JMGT} models
excitation, for example by a piezoelectric transducer array.  


We refer to \cite{AcostaUhlmannZhai:2022,
nonlinearity_imaging_Westervelt,nonlinearity_imaging_fracWest,nonlinearityimaging} for results related to the identification of the nonlinearity coefficient $\kappa$ alone in a classical model of nonlinear acoustics, the Westervelt equation (basically \eqref{JMGT} with vanishing relaxation time $\tau=0$).
In \cite{AcostaUhlmannZhai:2022} its uniqueness from the whole Neumann-Dirichlet map (instead of the single measurement \eqref{observation}) is shown; \cite{nonlinearityimaging} provides a uniqueness and conditional stability result for the linearized problem of identifying $\kappa$ in a higher order model of nonlinear acoustics in place of the  Westervelt equation.
In \cite{nonlinearity_imaging_Westervelt,nonlinearity_imaging_fracWest} we have proven injectivity of the linearized forward operator mapping $\kappa$ to $h$ in the Westervelt equation with classical strong damping and also with some fractional damping models as relevant in ultrasonics. 

A proof of injectivity of the linearized forward operator for the simultaneous recovery of $\kappa(x)$ and $c_0(x)$ in the Westervelt equation in dimension $d\in\{1,2,3\}$ from measurements with two excitations can be found in \cite{nonlinearity_imaging_both}, where it enables to apply a frozen Newton method and to show its convergence. For reconstruction results in 1-d we refer to \cite{nonlinearity_imaging_both} and in 2-d (based on a multiharmonic expansion for the Westervelt equation) to \cite{nonlinearity_imaging_2d}. 

A similar linearized uniqueness proof for JMGT in place of Westervelt will serve as a basis for proving local uniqueness of $\kappa$ by means of the Inverse Function Theorem here.
A considerable part of this paper is devoted to establishing the required well-definedness and differentiability of the forward operator.
Beyond this, the aim of this paper is to study simultaneous identification of $\kappa$ and $c_0$ or $b_0$ as space variable functions. Indeed, we will show that $\kappa(x)$ is locally unique even for unknown $c_0(x)$ and also provide results on simultaneous identifiability of $\kappa(x)$ and $c_0(x)$ or of $\kappa(x)$ and $b_0(x)$.

\subsection{The inverse problem}\label{sec:inverse}
The pointwise observation setting \eqref{observation} is an idealized one and will be extended to a more general observation operator $\obsop$, allowing generalization to, e.g., locally averaging observations, cf. \eqref{eqn:obs}.
Also the PDE model will be adapted to take into account the fact that for sufficiently small pressure amplitudes, nonlinearity can be neglected, cf. \eqref{eqn:JMGT_init_D_intro}.

Therewith we consider identification of the space dependent coefficients $\kappa(x)$, $c_0(x)$, $b_0(x)$, 
in the attenuated and switched JMGT equation in pressure form 
\begin{equation}\label{eqn:JMGT_init_D_intro}
\begin{aligned}
\tau u_{ttt}+\bigl((1-2\kappa\sigma u)u_t\bigr)_t+c^2\calAc u +\tau c^2\calAc u_t
+ b_0 (\tau u_t+u)_t &= r \quad t>0\\
u(0)=u_0, \quad u_t(0)=u_1, \quad u_{tt}(0)&=u_2,
\end{aligned}
\end{equation}
from observations
\begin{equation}\label{eqn:obs}
h(t)=\obsop u(t), 
\quad t\in(0,T).
\end{equation}
with some linear and bounded observation operator $\obsop:V\mapsto Y$, mapping from a space $V$ of $x$-dependent functions to some data space $Y$. Typical examples are
\[
\begin{aligned}
&\obsop:H^s(\Omega)\to L^2(\Sigma)&& v\mapsto \textup{tr}_\Sigma v \\
&\obsop:C(\overline{\Omega})\to\mathbb{R}^N && v\mapsto \Bigl(v(x_i)\Bigr)_{i\in\{1,\ldots,N\}}\\
&\obsop:L^2(\Omega)\to\mathbb{R}^N && v\mapsto \Bigl(\int_\Omega \eta_i(x) v(x)\, dx\Bigr)_{i\in\{1,\ldots,N\}}
\end{aligned}
\]
with the trace operator $\textup{tr}_\Sigma: H^s(\Omega)\to L^2(\Sigma)$, $s>\frac12$, measurement locations $x_i\in\overline{\Omega}$ and $\eta_i\in L^2(\Omega)$ locally supported around $x_i$.
In each of these examples, applicability of $\obsop$ to $u$ sets different spatial regularity requirements on the PDE solution.

In equation \eqref{eqn:JMGT_init_D_intro},
$c>0$ is the constant mean wave speed, and we make use of the combined elliptic operator
\begin{equation}\label{eqn:mathcalA}
\calAc=-\frac{c_0(x)^2}{c^2} \Delta \ \textup{ with }\  c_0,\ \frac{1}{c_0}\ \in L^\infty(\Omega)
\end{equation}
that contains the possibly spatially varying coefficient $c_0(x)>0$ and is equipped with the boundary conditions \eqref{impedance}. To achieve self-adjointness of $\calAc$, we use the weighted $L^2$ inner product with weight function $c^2/c_0^2(x)$, that is, 
$\|v\|_\Ltwow=\int_\Omega\frac{c^2}{c_0^2(x)}\, v^2(x)\, dx$.
We can think of $c_0$ as being spatially variable and of $\frac{c_0}{c}\sim 1$
to be normalized, while the magnitude of the wave speed is given by the
constant $c$.
Due to compactness of its inverse, the operator $\calAc$ has an
eigensystem $(\lambda_j,(\varphi_j^k)_{k\in K^{\lambda_j}})_{j\in\mathbb{N}}$
(where $K^{\lambda_j}$ is an enumeration of the eigenspace corresponding to
$\lambda_j$).
This allows for a diagonalization of the operator as
$(\calAc v)(x) = \sum_{j=1}^\infty \lambda_j\sum_{k\in K^{\lambda_j}} \langle v,\varphi_j^k\rangle \varphi_j^k(x)$.

\begin{remark}\label{rem:varrho}
Taking into account spatial variability of the mass density, \eqref{JMGT} would need to be written as 
$\tfrac{\tau}{\lambda(x)}u_{ttt}+\frac{1}{\lambda(x)}u_{tt}-\nabla\cdot(\frac{1}{\varrho(x)}\nabla u) -\tau\nabla\cdot(\frac{1}{\varrho(x)}\nabla u_t) +D[u] = \kappaintro(u^2)_{tt}+r$ with $u$ being the pressure, $\lambda$ the bulk modulus, $\varrho$ the mass density, and $c_0=\frac{\lambda}{\varrho}$ the sound speed, (cf., e.g., \cite{BambergerGlowinskiTran,LuckaPerezlivaTreebyCox2022} for the linear case). Usually $x$-dependence of $\varrho$ is not taken into account; Note that our general setting with a spatial differential operator $\calAc$ would be able to cover this dependency as well.
\end{remark}

While $\kappaintro$, $c_0$, $b_0$,  are positive, possibly space dependent coefficients, the relaxation time $\tau\geq0$ is constant and the switching factor $\sigma$ $\in[0,1]$ depends on the magnitude of the pressure and in particular vanishes for small $u$, see \eqref{eqn:sigma} for its precise definition.

\bigskip

It was shown in, e.g., \cite{nonlinearity_imaging_Westervelt,nonlinearity_imaging_fracWest}, that essential information on the space-dependent coefficients is contained in the residues at the poles of the Laplace transformed observations $\widehat{h}$. This, in its turn, via the elementary identity 
\begin{equation}\label{eqn:limitpoles}
\begin{aligned}
\textup{Res}(\widehat{h},p)&=\lim_{z\to p}(z-p)\widehat{h}(z)
=\lim_{z\to p}\widehat{h'-ph}(z)+h(0)\\
&=\lim_{T\to\infty}\lim_{z\to p}\int_0^T e^{-izt} (h'(t)-ph(t))\, dt +h(0)\\
&=\lim_{T\to\infty}\int_0^T \frac{d}{dt}\Bigl(e^{-ipt} h(t)\Bigr)\, dt +h(0)
=\lim_{T\to\infty} e^{-ipT} h(T),
\end{aligned}
\end{equation}
(provided absolute convergence holds, which allows us to interchange the integrals), leads us to study the asymptotics of solutions as time tends to infinity.

Correspondingly, in the uniqueness proofs of Section~\ref{sec:uniqueness} we will make use of the forward operator $\mathbb{F}$ that takes $\kappa$ to the residues of the Laplace transform $\widehat{\obsop u}$  at the poles of the Laplace transformed observation $\widehat{h}$.
In order to be able to take Laplace transforms, throughout Section~\ref{sec:uniqueness}  we will assume to have observations on the whole positive timeline, which in case of $r$ being analytic with respect to time (for example, just vanishing) from a time instance $\underline{T}$ on, follows by analytic continuation of the Fourier components of $u$, which will be shown to satisfy linear 
ODEs from a time instance $T_*$ on in Section~\ref{sec:uniqueness}.

\medskip

Continuous Fr\'{e}chet differentiability of the forward operator is an essential ingredient of the Inverse Function Theorem, that we plan to employ in our uniqueness proof. Concerning well-posedness of the JMGT equation we can largely rely on existing results in the literature, see, e.g. \cite{BongartiCharoenphonLasiecka20,BucciEller:2021,KLM12_MooreGibson,KLP12_JordanMooreGibson, fracJMGT,MarchandMcDevittTriggiani12,NikolicSaidHouari21a,
NikolicSaidHouari21b,pellicer2019,RackeSaid-Houari:2021}, in particular \cite{b2zeroJMGT} for the undamped case, which is relevant for the weak (rather than strong) damping we are employing here. Indeed, a proof of Fr\'{e}chet differentiability would fail for the Westervelt equation in the absence of strong damping, due to the loss of regularity arising in that case, see, e.g.,  \cite{DoerflerGernerSchnaubelt2016,b2zero}.
While well-posedness in the above mentioned references is inherently local in time in the above cited references, the switching allows us to prove global in time well-posedness of \eqref{eqn:JMGT_init_D_cutoff} in Section~\ref{sec:forward}, which is essential for the evaluation of residues according to \eqref{eqn:limitpoles}.

Section~\ref{sec:uniqueness} then contains several uniqueness results that rely on the previously shown differentiability and a proof of the linearized forward operator being an isomorphism in appropriately chosen topologies:
\begin{itemize}
\item local uniqueness of $\kappa$; 
\item linearized uniqueness of $\kappa$ and $c_0$;
\item linearized uniqueness of $\kappa$ and $b_0$;
\end{itemize}
from the single boundary observation \eqref{eqn:obs}.


\subsection*{Notation}
Below we will abbreviate $\Honetwo=H_0^1(\Omega)\cap H^2(\Omega)$ and make use of the spaces $\dot{H}^{s}(\Omega)$ induced by the norm 
\begin{equation}\label{sobolev_norm}
\|v\|_{H^s(\Omega)}=\Bigl(\sum_{j=1}^\infty \lambda_j^s\sum_{k\in K^{\lambda_j}}
 |\langle v,\varphi_j^k\rangle|^2\Bigr)^{1/2}
\end{equation}
with the eigensystem $(\lambda_j,\varphi_j)$ of the operator $\calAc$.

Moreover, the Bochner-Sobolev spaces $L^p(0,T;Z)$, $H^q(0,T;Z)$ with $Z$ some Lebesgue or Sobolev space and $T$ a finite or infinite time horizon will be used.
By $L^p(X)$, $W^{s,p}(X)$ we abbreviate the Bochner spaces over the whole positive real line $L^p(0,\infty;X)$, $W^{s,p}(0,\infty;X)$. 

$\mathcal{B}_\rho(x_0)^X=\{x\in X\, : \, \|x\|_X\leq \rho\}$ denotes the closed ball with radius $\rho>0$ and center $x_0$ in the normed space $X$.

We denote the Laplace transform of a function $v\in L^1(0,\infty)$ by $\widehat{v}(z)=\int_0^\infty e^{-zt}v(t)\, dt$ for all $z\in\mathbb{C}$ such that this integral exists.

The ordinary $\Ltwo$ inner product and norm will be denoted by $|\cdot|_\Ltwo$, $(\cdot,\cdot)_\Ltwo$, while the weigthed $L^2$-inner product with weight $\tfrac{c^2}{c_0^2}$ (or more generally, the inner product in which $\calAc$ is symmetric) will be abbreviated by $\langle,\cdot,\cdot\rangle$.

Generic constants will be denoted by $C$ (possibly indicating in parentheses dependence on certain quantities). Some specific constants are $C_{\textup{PF}}$ as appearing in the Poincar\'{e}-Friedrichs inequality 
$|v|_\Ltwo\leq C_{\textup{PF}}|\nabla|_\Ltwo$, $v\in H_0^1(\Omega)$, and the norm $C_{H^s,L^p}^\Omega$ of the embedding $H^s(\Omega)\to L^p(\Omega)$.

\section{The forward problem}\label{sec:forward}
In this section we analyze the forward problem: first of all, in Section~\ref{sec:S}, the 
parameter-to-state map $S:(\kappa,c_0^2,b_0)\mapsto u$; then, in Section~\ref{sec:F} the total forward operator, which we define by mapping the state $u$ obtained from $S$ into the observations \eqref{eqn:obs} and then further into the residues at the poles of $\widehat{h}$. The latter, along with an appropriately defined topology, will be the right setting for proving uniqueness of $\kappa$ by means of the Inverse Function Theorem.

\subsection{Well-definedness and differentiability of the parameter-to-state map}
\label{sec:S}

With a smooth cutoff function $\cutoff:[0,\infty)\to[0,1]$ such that for some thresholds $0<\underline{m}<\overline{m}<\infty$
\begin{equation}\label{eqn:cutoff}
\cutoff([0,\underline{m}])=\{0\}, \quad \cutoff([\bar{m},\infty))=\{1\}, 
\quad \cutoff\in W^{1,\infty}(0,\infty), \quad \cutoff'\geq0\mbox{ a.e.}
\end{equation}
and
\begin{equation}\label{eqn:sigma}
\sigma(t)=\cutoff(|\utau(t)|_\Ltwo^2).
\end{equation}
we consider \eqref{eqn:JMGT_init_D_intro}, that is
\begin{equation}\label{eqn:JMGT_init_D_cutoff}
\begin{aligned}
\tau u_{ttt}+\bigl((1-2\kappa\sigma u)u_t\bigr)_t+c^2\calAc u +\tau c^2\calAc u_t
+ b_0 (\tau u_t+u)_t &= r \quad t>0\\
u(0)=u_0, \quad u_t(0)=u_1, \quad u_{tt}(0)&=u_2,
\end{aligned}
\end{equation}
where for simplicity of exposition the operator $\calAc=-\frac{c_0^2}{c^2}\Delta$ is equipped with homogeneous Dirichlet boundary conditions (while the more general case \eqref{impedance} follows analogously). 
Moreover, in \eqref{eqn:JMGT_init_D_cutoff} we have positive space-dependent coefficients $\kappa$, $c_0$, $b_0$, constants $\tau>0$, $c>0$, as well as the time dependent function $\sigma $ as defined in \eqref{eqn:sigma}.
The wave energy 
\begin{equation}\label{E0}
\mathcal{E}_0[\utau](t) := \frac12\Bigl(|\utau_t(t)|_\Ltwo^2+c^2 |\nabla \utau(t)|_\Ltwo^2\Bigr)
\end{equation}
of $\utau=\tau u_t+t$ and therewith $|\utau(t)|_\Ltwo^2$ will be shown to decay to zero as $t\to\infty$.
Thus, the nonlinearity will be switched off after a certain time instance $T_*$ and therefore the solution follows a linear weakly damped wave equation from that time instance on, so that we have full control over the pole locations and residues of the Laplace transform of its solution.

An important observation here is that $T_*$ can be steered by the magnitude of $b_0$; more precisely, $T_*$ can be made smaller by making $b_0$ larger, while the local existence time $T$ of \eqref{eqn:JMGT_init_D_cutoff} does not decrease with increasing $b_0$. This allows us to combine the local in time nonlinear existence interval with the time span of linear wave propagation to obtain global in time well-posedness and differentiability results for  \eqref{eqn:JMGT_init_D_cutoff}.

\bigskip

We first of all state an adaption of existing local in time well-posedness results for \eqref{eqn:JMGT_init_D_cutoff} to the situation of spatially variable coefficients and weak damping, see subsections~\ref{sec:wellposed_lin}, \ref{sec:wellposed}.

Additionally, in subsection~\ref{sec:E0decay}, we obtain exponential decay of the wave energy $\mathcal{E}_0[\utau]$ due to the persistent weak damping, 
which then implies that nonlinearity is switched off after some time instance $T_*$ and $u$ follows a linear weakly damped wave equation. This is essential for controlling its residues that according to \eqref{eqn:limitpoles} are governed by the asymptotics as $t\to\infty$.

Finally, in subsection~\ref{sec:diff}, we prove continuous differentiability of the forward operator in a norm dictated by the uniqueness proof.

Throughout this section we assume $\Omega$ to be a bounded $C^{1,1}$ domain so that we can make use of elliptic regularity according to, e.g., \cite[Theorem 2.4.2.5]{Grisvard}.
 
\subsubsection{Energy estimates and well-posedness of the linearization}\label{sec:wellposed_lin}
We start by analyzing a linear version of \eqref{eqn:JMGT_init_D_cutoff}
\begin{equation}\label{eqn:JMGT_init_D_cutoff_lin}
\begin{aligned}
\tau u_{ttt}+\bigl((1-\bar{\sigma}\acoeff)u_t\bigr)_t+c^2\calAc u +\tau c^2\calAc u_t
+ b_0 (\tau u_t+u)_t -\mu u_t -\eta u&= \rcoeff \quad t>0\\
u(0)=u_0, \quad u_t(0)=u_1, \quad u_{tt}(0)&=u_2,
\end{aligned}
\end{equation}
with 
\begin{equation}\label{eqn:sigma_lin}
\begin{aligned}
&\bar{\sigma} \in W^{1,\infty}(0,\infty), \quad 
\bar{\sigma}([0,\infty))\subseteq[0,1].
\end{aligned}
\end{equation}

Later on, in Theorem~\ref{thm:wellposed} we will use Banach's Contraction Principle to prove convergence of a fixed point iteration defined by $u_{n+1}$ solving \eqref{eqn:JMGT_init_D_cutoff_lin} with 
$\bar{\sigma}=\cutoff(|u_n(t)|_\Ltwo^2)$, 
and $\acoeff=2\kappa u_n$ towards the unique solution of \eqref{eqn:JMGT_init_D_cutoff}.

By adapting the proof of \cite[Proposition 3.1]{b2zeroJMGT} we obtain
\begin{proposition}\label{prop:wellposed_lin}
Let $T>0$, $c_0,\frac{1}{c_0}$
$b_0\in L^\infty(\Omega)\cap W^{1,3}(\Omega)$, $b_0\geq0$, 
$\mu\in L^\infty(0,T;\Honetwo)$,
$\eta\in L^\infty(0,T;H_0^1(\Omega))$, $\bar{\sigma}\in W^{1,\infty}(0,T)$.

Then there exists $\rho_\acoeff>0$ such that for any 
$\acoeff\in L^\infty(0,T;L^\infty(\Omega))\cap L^\infty(0,T;W^{1,3}(\Omega))\cap W^{1,\infty}(L^2(\Omega))$ 
satisfying 
\begin{equation}\label{smallness_prop}
\begin{aligned}
&\|\acoeff\|_{L^\infty(0,T;L^\infty(\Omega))}
+\|\nabla\acoeff\|_{L^\infty(0,T;L^3(\Omega))}
+\|(\bar{\sigma}\acoeff)_t\|_{L^\infty(0,T;L^2(\Omega))}\\
&+\|\nabla b_0\|_{L^3(\Omega)}
+ \|\mu\|_{L^\infty(0,T;H^2(\Omega)}
+ \|\nabla\eta\|_{L^\infty(0,T;L^2(\Omega)}
\leq \rho_\acoeff,
\end{aligned}
\end{equation}
and any $\rcoeff\in L^2(0,T; H_0^1(\Omega)) \cap L^\infty(0,T ;L^2(\Omega))$, 
$(u_0, u_1, u_2) \in \Honetwo\times\Honetwo\times H_0^1(\Omega)$
the initial boundary-value problem \eqref{eqn:JMGT_init_D_cutoff_lin} has a unique solution
\begin{equation}\label{eqn:U}
u \in U_T= W^{3,\infty}(0,T;\Ltwo)\cap W^{2,\infty}(0,T;H_0^1(\Omega))\cap W^{1,\infty}(0,T;\Honetwo)
\end{equation}
on $(0,T)$.	
Furthermore, this solution satisfies
\begin{equation}\label{enest_prop}
\begin{aligned}
&\nLtwo{u_{ttt}(t)}^2+\nLtwo{\nabla u_{tt}(t)}^2+\nLtwo{\Delta u_t(t)}^2+\nLtwo{\Delta u(t)}^2 \\
& \leq \tilde{C}(\tau)(1+T)\Bigl(
\nLtwo{\utau_{tt}(t)}^2+\nLtwo{\nabla \utau_t(t)}^2+\nLtwo{\Delta \utau(t)}^2\Bigr)\\
& \leq C(\tau)e^{K(\tau)(\rho_\acoeff^2+1)T}
\Bigl(\nLtwo{\nabla u_2}^2 + \nLtwo{\Delta u_1}^2  + \nLtwo{\Delta u_0}^2  +\nLtwoLtwo{\nabla \rcoeff}^2\Bigr),
\end{aligned}
\end{equation}	
for all $t\in(0,T)$, where the constants $C(\tau)$ and $K(\tau)$ tend to $+\infty$ as $\tau\to0^+$, but are independent of $\rho_\acoeff$, $b_0$, and the final time $T$.
\end{proposition}

\begin{proof}
The proof relies on an energy estimate that is obtained by testing \eqref{eqn:JMGT_init_D_cutoff_lin}, that is, in terms of $\utau=\tau u_t+u$,
\begin{equation}\label{eqn:JMGT_init_D_cutoff_lin_z}
\begin{aligned}
&\utau_{tt}+c^2\calAc \utau +\tilde{b}_0 \utau_t -\tilde{\mu} \utau -\tilde{\eta} u = \rcoeff,\\
&\utau(0)=\tau u_1+u_0, \quad \utau_t(0)=\tau u_2+u_1,\\ 
&\text{where $\tilde{b}_0=b_0-\tfrac{1}{\tau}\bar{\sigma}\acoeff$, 
$\tilde{\mu}=\tfrac{1}{\tau}(\mu+(\bar{\sigma}\acoeff)_t-\tfrac{1}{\tau}\bar{\sigma}\acoeff)$,
$\tilde{\eta}= \eta-\tilde{\mu}$},
\end{aligned}
\end{equation}
with $-\Delta \utau_t=-\Delta (\tau \utau_t+\utau)_t$ and integrating by parts with respect to space in the first, third and fourth term, which results in the spatial differentiability requirements $\nabla\acoeff(t)$, $\nabla b_0$, $\nabla\acoeff_t(t)$ $\in L^3(\Omega)$ with sufficiently small norm.\\
The additional term $b_0 (\tau u_t+u)_t$ as compared to \cite[Proposition 3.1]{b2zeroJMGT} then clearly yields a nonnegative contribution to the left hand side.
\end{proof}

\subsubsection{Well-posedness of the switched JMGT equation}\label{sec:wellposed}
Analogously to \cite[Theorem 4.1]{b2zeroJMGT}, Banach's Contraction Principle yields local in time well-posedness of the nonlinear problem.

In view of the requirements \eqref{smallness_prop} with $\acoeff=2\kappa u$, we first of all consider the coefficient space
\begin{equation}\label{Xstrong_wellposed}
X_T=(L^\infty(\Omega)\cap W^{1,3}(\Omega))^3
\text{ for $(\kappa,c_0^2,b_0)$}. 
\end{equation}

\begin{theorem}\label{thm:wellposed}
Assume that $T>0$, $c>0$, $b>0$, $\rho>0$.

Then there exists $\rho_0>0$ such that for any $(\kappa,c_0^2,b_0)\in \mathcal{B}_\rho^{X_T}(0,c,b)$ and any excitation $r\in L^2(0,T; H_0^1(\Omega)) \cap L^\infty(0,T ;L^2(\Omega))$ as well as initial data 
$(u_0, u_1, u_2) \in \Honetwo\times\Honetwo\times H_0^1(\Omega)$
with 
\begin{equation}\label{smalldata}
\|\rcoeff\|_{L^2(0,T; H_0^1(\Omega)) \cap L^\infty(0,T ;L^2(\Omega))}^2
+|\nabla u_2|_\Ltwo^2+|\Delta u_1|_\Ltwo^2+|\Delta u_0|_\Ltwo^2 \leq\rho_0,
\end{equation}
there exists a unique solution $u\in U_T$ (as defined in \eqref{eqn:U}) of \eqref{eqn:JMGT_init_D_cutoff}
and this solution satisfies the energy estimate \eqref{enest_prop}.
\end{theorem}

\subsubsection{Exponential decay of the wave energy}\label{sec:E0decay}

By testing \eqref{eqn:JMGT_init_D_intro}, that is,
\[
\begin{aligned}
&\utau_{tt}-c_0^2\Delta \utau +\tilde{b}_0 \utau_t -\tilde{\mu} u_t = r,\\ 
&\utau=\tau u_t+u\,,\\
&\text{where $\tilde{b}_0=b_0-\tfrac{1}{\tau}2\kappa\sigma u$, 
$\tilde{\mu}=2\kappa(\tfrac{1}{\tau}\sigma u -(\sigma u)_t)$}, 
\end{aligned}
\]
with $\utau_t+\theta \utau=(\tau u_t+u)_t+\theta(\tau u_t+u)_t$ in place of $-\Delta \utau_t$, we obtain the energy identity 
\begin{equation}\label{enid0}
\begin{aligned}
&\tfrac12\tfrac{d}{dt}\Bigl[
|\utau_t|_\Ltwo^2
+|c_0\nabla \utau|_\Ltwo^2
+\theta|\tilde{b}_0^{1/2}\, \utau|_\Ltwo^2
+\theta(\utau_t,\utau)_\Ltwo
\Bigr]\\
&+|(\tilde{b}_0-\theta)^{1/2}\, \utau_t|_\Ltwo^2
+\theta |c_0\nabla \utau|_\Ltwo^2\\
&=(r-\tilde{\mu} u_t-\nabla \utau\cdot\nabla c_0^2,\utau_t+\theta \utau)_\Ltwo,
\end{aligned}
\end{equation}
provided $b_0\geq\theta+\tfrac{2}{\tau}\kappa\sigma u$.
Here, by Young's Inequality
\[
|\theta(\utau_t,\utau)_\Ltwo|\leq 
\tfrac12|\utau_t|_\Ltwo^2
+\tfrac12 \theta|\tilde{b}_0^{1/2}\, \utau|_\Ltwo^2
\]
since $\theta\leq \tilde{b}_0$ and 
\[
|(\nabla \utau\cdot\nabla c_0^2,\utau_t+\theta\utau)_\Ltwo|\leq
\tfrac12 |(\tilde{b}_0-\theta)^{1/2}\, \utau_t|_\Ltwo^2
+\tfrac12 \theta |c_0\nabla \utau|_\Ltwo^2
\]
if 
$4C_{PF}\|\tfrac{1}{c_0}\|_{L^\infty(\Omega)} \|\nabla c_0\|_{L^\infty(\Omega)}<1$,
$\tfrac{1}{\tilde{b}_0-\theta}\leq \tilde{\epsilon} \leq \tfrac{\theta(1-4C_{PF}\|\tfrac{1}{c_0}\|_{L^\infty(\Omega)} \|\nabla c_0\|_{L^\infty(\Omega)})}{4\|\nabla c_0\|_{L^\infty(\Omega)}^2}$
for some $\tilde{\epsilon}>0$.

This allows us to prove exponential decay of $\mathcal{E}_0[\utau](t)$ and thus of $|\utau(t)|_\Ltwo^2$. 
\begin{corollary}\label{cor:E0decay}
Under the assumptions of Theorem~\ref{thm:wellposed} with 
\begin{equation}\label{blarge}
\tfrac{2}{\tau}\|\kappa\|_{L^\infty(\Omega)} C_{H^2\to L^\infty}C(\tau)e^{K(\tau)(\rho_\acoeff^2+1)T}
\rho_0+\rho\leq \tfrac{b}{4}, 
\end{equation}
and
\begin{equation}\label{nablac0small}
4C_{PF}\|\tfrac{1}{c_0}\|_{L^\infty(\Omega)} \|\nabla c_0\|_{L^\infty(\Omega)}\leq\tfrac12
\end{equation}
there exist $\alpha=\alpha(b)>0$, $C>0$ such that for all $(\kappa,c_0^2,b_0)\in\mathcal{B}_\rho^X(0,c,b)$, 
$r\in L^2(0,T; H_0^1(\Omega)) \cap L^\infty(0,T ;L^2(\Omega))$, 
$(u_0, u_1, u_2) \in \Honetwo\times\Honetwo\times H_0^1(\Omega)$
 with \eqref{smalldata}, the estimate
\begin{equation}\label{enest_cor}
e^{\alpha t}\mathcal{E}_0[\utau](t) \leq \mathcal{E}_0[\utau](0) + C\alpha\|e^{\alpha/2\, \cdot} r\|_{L^2(0,t;\Ltwo)}^2
\quad\text{ for all }t\in(0,T)
\end{equation}
holds with some constant $C>0$ independent of $T$ and $b$. 

In particular, for $T_*=T_*(b):=\tfrac{1}{\alpha}\ln(\tfrac{C_{\textup{PF}}}{\underline{m}}\rho_1)$ and any $\kappa$, $c_0$, $b_0$, $r$, $u_0$, $u_1$ satisfying the above and 
\[
\mathcal{E}_0[u](0) + C\alpha\|e^{\alpha/2\, \cdot} r\|_{L^2(0,t;\Ltwo)}^2\leq\rho_1^2,
\] we have $\sigma[u](t)=0$ for all $t\geq T_*$. 

Here $\alpha(b)\to\infty$ and $T_*(b)\to0$ as $b\to\infty$.
\end{corollary}
\begin{proof}
Setting $\bar{\sigma}=\cutoff(|\utau(t)|_\Ltwo^2)$, $\acoeff=2\kappa u$, and using  \eqref{blarge}, we obtain $\tilde{b}_0\geq \tfrac{3b}{4}$. 
Thus, with \eqref{nablac0small}, choosing $\theta:=b/2$, and applying Young's inequality, we obtain from 
\eqref{enid0} that 
\[
\tfrac{d}{dt}\mathcal{E}_0[\utau](t)+\tfrac{b}{4}\mathcal{E}_0[\utau](t)\leq C b |r(t)|_\Ltwo^2,
\]
with $C$ independent of $b$.
Hence 
\[
\begin{aligned}
&\tfrac{d}{dt}\Bigl[e^{(b/4) t}\mathcal{E}_0[\utau](t)\Bigr]
=e^{(b/4) t}\Bigl\{\tfrac{d}{dt}\mathcal{E}_0[\utau](t)+\tfrac{b}{4}\mathcal{E}_0[\utau](t)\Bigr\}
\leq C b e^{(b/4) t} |r(t)|_\Ltwo^2
\end{aligned}
\]
After integration over $[0,t]$, this implies the assertion with $\alpha(b)=b/4$.
\end{proof}

Hence, for fixed $T$, $\rho_0$, $\rho_1$, $\rho>0$, choosing $b$ sufficiently large, we can achieve that nonlinearity is switched off at $T_*<T$ and the PDE turns into a linear one (with still spatially variable coefficients) from $T_*$ on, which allows to conclude its global in time well-posedenss.
Thus, the assertions of Theorem~\ref{thm:wellposed} and of Proposition~\ref{prop:wellposed_lin} with $\acoeff=2\kappa\sigma u$ and $\sigma$ as in \eqref{eqn:sigma} remain valid with $T$ replaced by $\infty$ and $U_T$ replaced by $U_\infty$.

\subsubsection{Differentiability of the parameter-to-state map}\label{sec:diff}

Unter the assumptions 
\begin{equation}\label{condru0u1u2}
\begin{aligned}
&r\in L^2(H^1(\Omega))\cap L^\infty(L^2(\Omega)), \ e^{b/8\, \cdot} r\in L^2(0,t;\Ltwo),\\
&u_0,u_1\in \Honetwo, u_2\in H_0^1(\Omega)\text{ with \eqref{smalldata},}
\end{aligned}
\end{equation}
the parameter-to-state map 
\begin{equation}\label{S}
S:\mathcal{B}^X_\rho(0)\to U_\infty, \quad(\kappa,c_0^2,b_0)\mapsto u\ \text{ where $u$ solves \eqref{eqn:JMGT_init_D_cutoff}}
\end{equation} 
is well defined as a mapping from $X$ to $U:=U_\infty$ (cf. \eqref{eqn:U}), due to the above extension of Theorem~\ref{thm:wellposed} to the whole positive time line.
Slightly increasing the requirement coefficient regularity as compared to \eqref{Xstrong_wellposed} in view of \eqref{nablac0small} by setting 
\begin{equation}\label{Xstrong}
\begin{aligned}
&X=(L^\infty(\Omega)\cap W^{1,3}(\Omega))\times W^{1,\infty}(\Omega)\times
(L^\infty(\Omega)\cap W^{1,3}(\Omega))
\text{ for $(\kappa,c_0^2,b_0)$},\\
&X=\Bigl(L^\infty(\Omega)\cap W^{1,3}(\Omega)\Bigr)^2
\text{ for $(\kappa,b_0)$}. 
\end{aligned}
\end{equation}
we separately consider the option of regarding $c_0^2$ as fixed and varying $\kappa$ and $b_0$ only, which in fact will make a difference in the spaces in which we will obtain differentiability.

Its formal derivative (which we will show to be the Fr\'{e}chet derivative in Proposition~\ref{prop:S})  is given by $S'(\kappa,c_0^2,b_0)(\dkappa,\dc_0^2,\db_0) = v$ where $v$ solves
\begin{equation}\label{PDESprime}
\begin{aligned}
&\tau v_{ttt}+\bigl((1-2\kappa\sigma[u]u)v_t\bigr)_t-c_0^2\Delta v-\tau c_0^2\Delta v_t + b_0 (\tau v_{tt}+ v_t) = \rcoeff(v)\\
&\text{with }\rcoeff(v):=
2\bigl(\kappa(\sigma[u]v +\sigma'[u][v])u+\dkappa\sigma[u] u) u_t\bigr)_t
+ \dc_0^2 \Delta \utau
- \db_0 \utau_t
\end{aligned}
\end{equation}
with homogeneous initial and boundary conditions.
Here and in the following we explicitely indicate dependence of $\sigma$ on $u$ by the notation $\sigma[u]$ and have 
\[
\sigma'[u][v](t) =\cutoff'(|\utau(t)|_\Ltwo^2) (\utau(t),v^\tau(t))_\Ltwo
\ \text{ with }v^\tau=\tau v_t+v.
\]
Well-posedeness of \eqref{PDESprime} can be shown by a contraction argument for the (affinely linear) fixed point operator $\mathcal{T}:\tilde{v}\mapsto v$ solving \eqref{PDESprime} with $\rcoeff(\tilde{v})$ in place of $\rcoeff(v)$.\footnote{Note that Banach's Contraction Principle in this affinely linear setting boils down to the use of the Neumann series.}
To this end, we use the fact that 
\[
\|\rcoeff(\tilde{v}_1)-\rcoeff(\tilde{v}_2)\|_{L^2(\Ltwo)}
\leq \tilde{C} R \normE{\tilde{v}_1-\tilde{v}_2}, 
\]
for some constant $\tilde{C}>0$, and 
\[
\normE{v}:=|\sqrt{\mathcal{E}_0[\tau v_t+v]}|_{L^\infty(0,\infty)}
\]
which by an estimate analogous to \eqref{enest_cor} yields
\[
\normE{\mathcal{T}\tilde{v}_1-\mathcal{T}\tilde{v}_2]}
\leq C \tilde{C} R \normE{\tilde{v}_1-\tilde{v}_2]},
\]
thus, by smallness of $R$, contractivity of $\mathcal{T}$.

Note that since the term $\dc_0^2 \Delta \utau$ is not necessarily contained in $L^2(H^1(\Omega))$, which would be needed for an application of Proposition~\ref{prop:wellposed_lin}, 
we had to resort to the weaker topology induced by the wave energy, that is,
\begin{equation}\label{Ulo}
U_{\textup{lo}}=\{u\in L^2(\Ltwo)\, : \, \tau u_t+u\in W^{1,\infty}(L^2(\Omega))\cap L^\infty(H_0^1(\Omega))\}
\end{equation}
and thus obtain existence and uniqueness of $v=S'(\kappa,c_0^2,b_0)(\dkappa,\dc_0^2,\db_0)$  in $U_\textup{lo}$ for $(\kappa,c_0^2,b_0)\in X$, $(\dkappa,\dc_0^2,\db_0)\in L^\infty(\Omega)^3$.

When considering the parameter-to-state-map $S(\kappa,b_0)$ needed in the recovery of $\kappa$ and $b_0$ only (while keeping $c_0^2$ fixed), this term is not present and 
\eqref{enest_prop} yields
\[
\|\mathcal{T}\tilde{v}_1-\mathcal{T}\tilde{v}_2]\|_U
\leq C \tilde{C} R \|\tilde{v}_1-\tilde{v}_2]\|_U,
\]
hence existence and uniqueness of $v=S'(\kappa,b_0)(\dkappa,\db_0)$ in the solution space $U$ for $(\dkappa,\db_0)\in L^\infty(\Omega)^2$.

Analogously we obtain a Lipschitz estimate of $S$
\begin{equation}\label{SLipschitz}
\begin{aligned}
&\|S(\tilde{\kappa},\tilde{c}_0^2,\tilde{b}_0)-S(\kappa,c_0^2,b_0)\|_{U_{\textup{lo}}}
\leq L\|(\tilde{\kappa},\tilde{c}_0^2,\tilde{b}_0)-(\kappa,c_0^2,b_0)\|_{L^\infty(\Omega)^3},\\
&\|S(\tilde{\kappa},\tilde{b}_0)-S(\kappa,b_0)\|_U 
\leq L\|(\tilde{\kappa},\tilde{b}_0)-(\kappa,b_0)\|_{L^\infty(\Omega)^2},
\end{aligned}
\end{equation}
which allows to extend $S$ to a (Lipschitz continuous) mapping 
$S:\mathcal{B}_\rho(0)^{\tilde{X}}\to U_{\textup{lo}}$ or $U$, respectively, where 
\footnote{Note that $L^\infty(\Omega)\cap W^{1,3}(\Omega)$ is not dense in $L^\infty(\Omega)$ for otherwise density of $C(\Omega)$ in $W^{1,3}(\Omega)$ would imply $C(\Omega)=L^\infty(\Omega)$.}
\begin{equation}\label{X}
\tilde{X}=\left(\overline{L^\infty(\Omega)\cap W^{1,3}(\Omega)}^{L^\infty}\right)^3 \text{ with }
\|\cdot\|_{\tilde{X}} = \|\cdot\|_{L^\infty(\Omega)^3}.
\end{equation}

To prove that $S'$ is indeed the Fr\'{e}chet derivative of $S$, we consider the PDE that is satisfied (along with homogeneous initial and boundary conditions) by the first order Taylor remainder 
$w=S(\tilde{\kappa},\tilde{c}_0^2,\tilde{b}_0)-S(\kappa,c_0^2,b_0)-S'(\kappa,c_0^2,b_0)((\tilde{\kappa},\tilde{c}_0^2,\tilde{b}_0)-(\kappa,c_0^2,b_0))
=\tilde{u}-u-v$, namely
\begin{equation}\label{PDETaylor}
\tau w_{ttt}+\bigl((1-2\kappa\sigma[u]u)w_t\bigr)_t-c_0^2\Delta w-\tau c_0^2\Delta w_t + b_0 (\tau w_t+w) = \rcoeff(\tilde{u},u,v)
\end{equation}
where $\tilde{u}=S(\tilde{\kappa})$, $u=S(\kappa)$, $v=S'(\kappa)(\tilde{\kappa}-\kappa)$, 
\begin{equation}\label{rTaylor}
\begin{aligned}
&\rcoeff(\tilde{u},u,v)\\
&:=
2\Bigl(\bigl(\tilde{\kappa}\sigma[\tilde{u}]\tilde{u}-\kappa\sigma[u]u\bigr)\tilde{u}_t
-\bigl(\kappa(\sigma[u]v +\sigma'[u][v]u)+(\tilde{\kappa}-\kappa)\sigma[u] u\bigr) u_t
\Bigr)_t\\
&\quad+c_0^2\Delta w^\tau-(\tilde{c}_0^2\Delta\tilde{u}^\tau -c_0^2\Delta\utau -c_0^2\Delta v^\tau -(\tilde{c}_0^2-c_0^2)\Delta u^\tau)\\
&\quad+b_0 w^\tau_t-(\tilde{b}_0\tilde{u}^\tau_t -b_0\utau_t -b_0 v^\tau_t-(\tilde{b}_0-b_0) u^\tau_t)\\
&=
2 \Bigl((\tilde{\kappa}-\kappa)(\sigma[\tilde{u}]\tilde{u}-\sigma[u]u)\tilde{u}_t 
+\kappa(\sigma[\tilde{u}]\tilde{u}-\sigma[u]u)(\tilde{u}-u)_t\\
&\qquad\qquad\quad+\kappa\bigl(\sigma[\tilde{u}]\tilde{u}-\sigma[u]u-\sigma[u]v -\sigma'[u][v]u\bigr)u_t
\Bigr)_t\\
&\quad
-(\tilde{c}_0^2-c_0^2) \Delta(\tilde{u}^\tau-u^\tau)
-(\tilde{b}_0-b_0) (\tilde{u}^\tau_t-u^\tau_t)
\end{aligned}
\end{equation}
and 
\[
\begin{aligned}
&\sigma[\tilde{u}]\tilde{u}-\sigma[u]u
=(\sigma[\tilde{u}]-\sigma[u])\tilde{u}+\sigma[u](\tilde{u}-u)\\[1ex]
&\sigma[\tilde{u}]\tilde{u}-\sigma[u]u-\sigma[u]v -\sigma'[u][v]u\\
&=(\sigma[\tilde{u}]-\sigma[u])(\tilde{u}-u)
+\sigma[u]
w
+(\sigma[\tilde{u}]-\sigma[u]-\sigma'[u][\tilde{u}-u])u
+\sigma'[u][
w
])u.
\end{aligned}
\]
The energy estimates \eqref{enest_cor}, \eqref{enest_prop} yield 
$\|w\|_{U_{\textup{lo}}}$ $=O(\|(\tilde{\kappa},\tilde{c}_0^2,\tilde{b}_0)-(\kappa,c_0^2,b_0)\|_{L^\infty(\Omega)^3}^2)$ or $\|w\|_U$ $=O(\|(\tilde{\kappa},\tilde{b}_0)-(\kappa,b_0)\|_{L^\infty(\Omega)^2}^2)$ and 
analogously we obtain a Lipschitz estimate of $S'$
\begin{equation}\label{SprimeLipschitz}
\begin{aligned}
&\|S'(\tilde{\kappa},\tilde{c}_0^2,\tilde{b}_0)-S'(\kappa,c_0^2,b_0)\|_{L(\tilde{X},U_{\textup{lo}})}
\leq L\|(\tilde{\kappa},\tilde{c}_0^2,\tilde{b}_0)-(\kappa,c_0^2,b_0)\|_{L^\infty(\Omega)^3},\\
&\|S'(\tilde{\kappa},\tilde{b}_0)-S'(\kappa,b_0)\|_{L(\tilde{X},U)} 
\leq L\|(\tilde{\kappa},\tilde{b}_0)-(\kappa,b_0)\|_{L^\infty(\Omega)^2}.
\end{aligned}
\end{equation}
Altogether we have proven
\begin{proposition}\label{prop:S}
Assume that $c>0$, $\rho>0$, $b>0$, the latter being sufficiently large.

Then there exists $\rho_0>0$ such that for any $r$, $(u_0, u_1, u_2)$ satisfying \eqref{condru0u1u2}, the operator $S$ is well-defined as a mapping $\mathcal{B}_\rho^X(0,c^2,b)\to U=U_\infty$, cf. \eqref{eqn:U}. 
If additionally $\cutoff\in W^{2,\infty}(0,\infty)$ then $S$ satisfies the Lipschitz estimate \eqref{SLipschitz} and can be extended to a (Lipschitz continuous) mapping $S:\mathcal{B}_{\tilde{\rho}}^{\tilde{X}}(0,c^2,b)\to U_{\textup{lo}}$, which is Fr\'{e}chet differentiable with Lipschitz continuous derivative \eqref{SprimeLipschitz}.
The assertions remain valid with $U_{\textup{lo}}$ replaced by $U$ if $S$ is considered as a function of $(\kappa,b_0)$ only.
\end{proposition}
Since Fr\'{e}chet differentiability remains valid under strengthening of the topology in preimage space, the assertions of Proposition~\ref{prop:S} remain valid if we simply stay with the stronger space \eqref{Xstrong} in place of \eqref{X}.

\subsection{Well-definedness and differentiability of the forward operator for fixed $c_0(x)$ and constant $b_0(x)=b$}
\label{sec:F}
As a preparation for the use of the Inverse Function Theorem for proving uniqueness of $\kappa$, we will now study the operator $\mathbb{F}$ that maps $(\kappa,c_0^2,b_0)$ to the residues of 
$\widehat{\obsop u}$ 
at the poles of $\widehat{h}$. 

In our uniqueness proof, in order to take residues at poles $p$ of the Laplace transformed solution to the PDE and its linearization, we need sufficiently fast decay as time tends to infinity according to the identity \eqref{eqn:limitpoles}
\[
\textup{Res}(\widehat{f},p)=\lim_{T\to\infty} e^{-ipT} f(T).
\]
Note that the poles we consider will be single and their real parts will be negative, so that $|e^{-pt}|$ is exponentially increasing as $t\to\infty$. 
In particular, for an integrable function with finite support $[0,T]$ in time, the residue vanishes (which is clear in view of the fact that the Laplace transform of a finitely supported function has no poles). \\

In view of this, 
we will study the large time behavior of solutions.
Due to the exponential decay according to Corollary~\ref{cor:E0decay} and the switching by means of $\cutoff$, from $T_*$ on, the solution $u$ to \eqref{eqn:JMGT_init_D_cutoff} coincides with the (time shifted) solution $u_*$  of a linear PDE
\begin{equation}\label{uustar}
u(t-T_*)=u_*(t) \quad t>T_*,
\end{equation}
where 
\begin{equation}\label{eqn:JMGT_init_D_linconst}
\begin{aligned}
\utau_{*tt}+c^2\calAc \utau_* + b_0 \utau_{*t}
&= r_* \quad t>0\\
\utau_*(0)=\tau u_{*1}+u_{*0}, \quad \utau_{*t}(0)&=\tau u_{*2}+u_{*1}\\
\utau=\tau u_t+u
\end{aligned}
\end{equation}
In order to be able to apply separation of variables, we assume the (possibly space dependent) sound speed $c_0(x)$ to be fixed and $b_0$ to be constant $b_0(x)\equiv\bconst$.
\footnote{Sligtly more generally, we could assume the three operators $\text{id}$, $\calAc$, $b_0\cdot$ to be simultaneously diagonizable}
By means of an eigensystem $(\lambda_j,(\varphi_j^k)_{k\in K^{\lambda_j}})_{j\in\mathbb{N}}$ (where $K^{\lambda_j}$ is an enumeration of the eigenspace corresponding to $\lambda_j$) of $\calAc $, we can then write  $\utau_*(x,t)=\sum_{j=1}^\infty \sum_{k\in K^{\lambda_j}} u_{*j}^{\tau k}(t) \varphi_j^k(x)$ with $u_{*j}^{\tau k}(t)=\langle \utau_*(t),\varphi_j^k\rangle$ solving the relaxation equation 
\[
\begin{aligned}
{u_{*j}^{\tau k}}''+c^2\lambda_j u_{*j}^{\tau k} + \bconst {u_{*j}^{\tau k}}' &= r_{*j}^k \quad t>0\\
u_{*j}^{\tau k}(0)=u_{*j0}^{\tau k}, \quad {u_{*j}^{\tau k}}'(0)&=u_{*j1}^{\tau k},
\end{aligned}
\]
whose Laplace transformed solutions are given by
\begin{equation}\label{eqn:what}
\widehat{u}_{*j}^{\tau k}(z)=\frac{\widehat{r}_{*j}^k(z)+u_{*j1}^k+(z+\bconst) u_{*j0}^k}{\omega_{\lambda_j}(z)},\ \textup{ with } 
\omega_{\lambda}(z)=  
z^2 + 
\bconst
z + c^2\lambda\,. 
\end{equation}

The poles and residues of the resolvent functions $\frac{1}{\omega_{\lambda_j}(z)}$ compute explicitly as 
\begin{equation}\label{eqn:poles}
\begin{aligned}
p_j^\pm &= -\frac{\bconst}{2}\pm \imath\sqrt{c^2\lambda_j-\frac{\bconst^2}{4}}\,, \\
\textup{Res}(\tfrac{1}{\omega_{\lambda_j}};p_j^+)&=\lim_{z\to p_j^+}\frac{z-p_j^+}{(z-p_j^+)(z-p_j^-)}
=\frac{1}{\imath\sqrt{4c^2\lambda_j-\bconst^2}}\,.
\end{aligned}
\end{equation}
In particular, we have the following two essential properties
\begin{lemma}\label{lem:what}
The poles of $\frac{1}{\omega_\lambda}$ differ for different $\lambda$. 
\end{lemma}
\begin{lemma}\label{lem:resnon0}
The residues of the poles of $\frac{1}{\omega_\lambda}$ do no vanish. 
\end{lemma}
These properties remain valid for more general damping models involving fractional derivatives, see  \cite[Lemmas 11.4, 11.5]{BBB}. 

\begin{remark}
From the formula for the poles and residues it becomes evident why strong damping needs to be switched off eventually. Indeed,  with
\[
u_{*tt}+c^2\calAc  u_* + b_0 u_{*t} + b_1\calAc  u_{*t} 
\]
in place of \eqref{eqn:JMGT_init_D_linconst} we would have
\[
\begin{aligned}
p_j^\pm &= -\frac{b_0+b_1\lambda_j}{2}\pm \sqrt{\frac{(b_0+b_1\lambda_j)^2}{4}-c^2\lambda_j}\,, \\
\textup{Res}(\tfrac{1}{\omega_{\lambda_j}};p_j^+)&
=\frac{1}{2\sqrt{\frac{(b_0+b_1\lambda_j)^2}{4}-c^2\lambda_j}}\,.
\end{aligned}
\]
in place of \eqref{eqn:poles}.
Since $\lambda_j\to\infty$ as $j\to\infty$, the poles $p^+_j$ (up to finitely many) would be real and negative; in case $b_1>0$ their real parts would tend to $-\infty$ as $j\to\infty$.
\end{remark}

The explicit representation \eqref{eqn:poles} together with \eqref{uustar}, \eqref{eqn:limitpoles} allows us to compute the residues of the Laplace transformed components $u_j^{\tau k}=\langle \utau,\varphi_j^k\rangle$ at the poles $p_j$ as follows.
Since 
\[
\begin{aligned}
&u_j^{\tau k}(t)={\rm 1\!\!I}_{[0,T_*)}(t)u_j^{\tau k}(t)+ {\rm 1\!\!I}_{[T_*,\infty)}(t)u_{*j}^{\tau k}(t-T_*)=:u_{[0,T_*)}+u_{[T_*,\infty)}\\
&\textup{where }\widehat{u_{[0,T_*)}}(z)=0\textup{ and }
\widehat{u_{[T_*,\infty)}}(z)=\int_{T_*}^\infty e^{-zt} u_{*j}^{\tau k}(t-T_*)\,dt\\
&\hspace*{3cm}=\int_0^\infty e^{-z(s+T_*)} u_{*j}^{\tau k}(s)\,ds =
e^{-zT_*}\widehat{u}_{*j}^k(z)
\end{aligned}
\]
due to \eqref{eqn:limitpoles}, \eqref{eqn:what}, \eqref{eqn:poles}, we have 
\begin{equation}\label{eqn:Resjk}
\begin{aligned}
\textup{Res}(\widehat{u_j^{\tau k}};p_j^+)&=0+e^{-p_j^+T_*}\textup{Res}(\widehat{u}_{*j}^{\tau k};p_j^+)
= \frac{e^{-p_j^+T_*}(\widehat{r}_{*j}^k(p_j^+)+u_{*j1}^{\tau k}+ (p_j^+ +\bconst) u_{*j0}^{\tau k})}{\imath\sqrt{4c^2\lambda_j-\bconst^2}}
\end{aligned}
\end{equation}
In here, have 
\begin{equation}\label{eqn:Resjk_rhs}
\begin{aligned}
|e^{-p_j^+T_*}\widehat{r}_{*j}^k(p_j^+)|&
=|e^{-p_j^+T_*}\int_0^\infty e^{-p_j^+s}r_j^k(s+T_*)\, ds|
= |\int_{T_*}^\infty e^{-p_j^+t}r_j^k(t)\, dt|\\
&\leq \int_{T_*}^\infty e^{-\Re(p_j^+)t} |r_j^k(t)|\, dt
= \|e^{(\bconst/2)\cdot}r_j^k\vert_{[T^*,\infty)}\|_{L^1(0,\infty)}\,.
\end{aligned}
\end{equation}
Moreover with $u_{*j0}=\langle u(T_*),\varphi_j^k\rangle$, $u_{*j1}=\langle u_t(T_*),\varphi_j^k\rangle$ according to \eqref{uustar}, we have 
\begin{equation*}
\begin{aligned}
|e^{-p_j^+T_*}(u_{*j1}^{\tau k}+ p_j^+ u_{*j0}^{\tau k})|&
\leq e^{-\Re(p_j^+)T_*}\bigl(|\langle \utau_t(T_*),\varphi_j^k\rangle| + |p_j^+ +\bconst|\, |\langle \utau(T_*),\varphi_j^k\rangle|\bigr)\\
&\leq e^{(\bconst/2)T_*} \bigl(|\langle \utau_t(T_*),\varphi_j^k\rangle| + c\sqrt{2\lambda_j}\|\langle \utau(T_*),\varphi_j^k\rangle| \bigr)\\
&\leq 2 e^{(\bconst/2)T_*} \sqrt{\mathcal{E}_0[\langle \utau,\varphi_j^k\rangle\varphi_j^k](T_*)}.
\end{aligned}
\end{equation*}
This contains the wave energy of $\utau$ as it has evolved nonlinearly up to $T_*$ according to \eqref{eqn:JMGT_init_D_cutoff}.
To estimate it, we employ Corollary~\ref{cor:E0decay} from section \ref{sec:S}.

Thus with 
\begin{equation}\label{eqn:sqrt_clambda}
\left|\imath\sqrt{4c^2\lambda-\bconst^2}\right|\geq c\sqrt{\lambda} \quad \textup{ for all }\lambda\geq\frac{\bconst^2}{3c^2}
\end{equation}
in \eqref{eqn:Resjk}
we have shown the following.

\begin{proposition}\label{prop:boundresid}
Under the conditions of Theorem~\ref{thm:wellposed}, Corollary~\ref{cor:E0decay}, the solution to \eqref{eqn:JMGT_init_D_cutoff} exists for all $t>0$ and the generalized Fourier coeffcients of its Laplace transform satisfy the estimate
\[
\begin{aligned}
&\sum_{j=1}^\infty\lambda_j\sum_{k\in K^{\lambda_j}}\textup{Res}(\widehat{u_j^{\tau k}};p_j^+)^2\\
&\leq Ce^{(\bconst-\alpha)T_*}\Bigl(\mathcal{E}_0[u](0) + \|e^{\alpha/2\, \cdot} r\|_{L^2(0,T_*;\Ltwo)}\Bigr)
+ \|e^{(\bconst/2)\cdot}r\vert_{[T^*,\infty)}\|_{L^1(\Ltwo)}^2
\end{aligned}
\]
on $u_j^{\tau k}(t)=\langle \utau(t),\varphi_j^k\rangle$. 
Here $C>0$, $\alpha$, $T_*$ can be chosen independent of the individual inital data and excitation but depending only on the radii $\rho$, $\rho_0$, $\rho_1$ in Corollary~\ref{cor:E0decay}.
\end{proposition}

To prove differentiability with respect to $\kappa$ at any point $(\kappa,c_0^2,\bconst)\in \mathcal{B}_\rho^X(0,c^2,b)$, with $\nabla\bconst=0$, we write the forward operator in two alternative ways:
\[
\mathbb{F}=\mathcal{R}_{c_0^2,\bconst}\circ \obsop\circ S =\mathbb{A}_{c_0^2,\bconst}\circ S_* +\vec{a}_{c_0^2,\bconst}
\]
where $\obsop$ and $S$ are defined as in \eqref{eqn:obs}, \eqref{S} and 
\[
\mathcal{R}_{c_0^2,\bconst}:h\mapsto 
\Bigl((\textup{Res}(\widehat{h};p_m^+))_{m\in\mathbb{N}},
(\textup{Res}(\widehat{h};p_m^-))_{m\in\mathbb{N}}\Bigr)
\]
\begin{equation}\label{Sstar}
S_*:X\to H^2(\Omega)\times H_0^1(\Omega), \quad (\kappa,c_0^2,b_0)\mapsto (\utau(T_*),\utau_t(T_*)),
\end{equation}
(where attainment of $(\utau(T_*),\utau_t(T_*))$ in $H^2(\Omega)\times H_0^1(\Omega)$ can be justified analogously to attainment of initial data in the well-posedness proof of JMGT, see e.g., \cite{b2zero}),
\[
\begin{aligned}
\vec{a}_{c_0^2,\bconst} = \Bigl(&(\textup{Res}(\tfrac{1}{\omega_{\lambda_m}};p_m^+) \widehat{r}_{*m}(p_m^+)
B_{\lambda_m}[(1,\ldots,1)])_{m\in\mathbb{N}},\\
&\textup{Res}(\tfrac{1}{\omega_{\lambda_m}};p_m^-) \widehat{r}_{*m}(p_m^-)
B_{\lambda_m}[(1,\ldots,1)])_{m\in\mathbb{N}}\Bigr)
\end{aligned}
\]
where $\textup{Res}(\tfrac{1}{\omega_{\lambda_m}};p_m^+)=-\textup{Res}(\tfrac{1}{\omega_{\lambda_m}};p_m^-)= \frac{1}{\imath\sqrt{4c^2\lambda_m-\bconst^2}}$
\[
\begin{aligned}
\mathbb{A}_{c_0^2,\bconst}:(\utau_{*0},\utau_{*1})\mapsto 
\Bigl(&(\textup{Res}(\tfrac{1}{\omega_{\lambda_m}};p_m^+) 
B_{\lambda_m}[(u_{*m1}^{\tau k}+ (p_j^+ +\bconst) u_{*m0}^{\tau k})_{k\in K^{\lambda_m}})_{m\in\mathbb{N}},\\
&\textup{Res}(\tfrac{1}{\omega_{\lambda_m}};p_m^-) 
B_{\lambda_m}[(u_{*m1}^{\tau k}+ (p_j^+ +\bconst) u_{*m0}^{\tau k})_{k\in K^{\lambda_m}}])_{m\in\mathbb{N}}\Bigr)
\end{aligned}
\]
\begin{equation}\label{Blambda}
B_\lambda: \mathbb{R}^{|K^\lambda|}\to E^\Sigma_\lambda:=\mbox{span}((\obsop\varphi_k)_{k\in K^\lambda}), \quad (\beta_k)_{k\in K^\lambda}\mapsto 
\sum_{k\in K^\lambda} \beta_k \obsop\varphi_k.
\end{equation}
Here the operators $\mathcal{R}_{c_0^2,\bconst}$, $\mathbb{A}_{c_0^2,\bconst}$, $\vec{a}_{c_0^2,\bconst}$ and also $B_{\lambda_m}$ depend on $c_0^2,\bconst$ via dependence of the eigenfunctions and poles $p_m^\pm$ on $c_0^2,\bconst$.

To resolve the contributions within each eigenspace in our uniqueness proof, as in \cite{BBB} we impose a linear independence assumption on the eigenspace $E^\Sigma_\lambda$ of each eigenvalue $\lambda$
\begin{equation}\label{eqn:ass_inj_Sigma_rem}
\left(\sum_{k\in K^\lambda} \beta_k \varphi_k(x) = 0 \ \mbox{ for all }x\in\Sigma\right)
\ \Longrightarrow \ \left(\beta_k=0 \mbox{ for all }k\in K^\lambda\right)\,.
\end{equation}
That is, for each eigenvalue $\lambda$ of $\calAc $ with eigenfunctions $(\varphi_k)_{k\in K^\lambda}$, 
the eigenfunctions, when restricted to the observation manifold $\Sigma$ (or more generally, when $\obsop$ is applied) keep their linear independence. This is a similar to usual conditions imposed in coefficient identification problems in wave type equations that combine assumptions on the observation geometry with a non-trapping condition on the sound speed. 
We refer to, e.g., \cite{nonlinearity_imaging_fracWest} for a discussion.

This implies that for each eigenvalue $\lambda$, the mapping $B_\lambda$ defined by \eqref{Blambda}
is bijective and (as a mapping between finite dimensional spaces) its inverse is bounded.
We use this to define the norm in image space as 
\[
\|(\widehat{h}^+,\widehat{h}^-)\|_{\mathbb{Y}}^2:= 
\sum_{m=1}^\infty\lambda_m^2 
\|B_{\lambda_m}^{-1}\mbox{Proj}^Y_{E^\Sigma_{\lambda_m}} \widehat{h}_m^+\|_{\mathbb{R}^{K^{\lambda_m}}}^2
+\sum_{m=1}^\infty\lambda_m^2 
\|B_{\lambda_m}^{-1}\mbox{Proj}^Y_{E^\Sigma_{\lambda_m}} \widehat{h}_m^-\|_{\mathbb{R}^{K^{\lambda_m}}}^2.
\]
With the so defined space $\mathbb{Y}$, the operator $\mathbb{A}_{c_0^2,\bconst}:H^2(\Omega)\times H_0^1(\Omega)\to\mathbb{Y}$ is linear and bounded and therefore, due to differentiability of $S_*$ (which is shown analogously to differentiability of $S$), the forward operator $\mathbb{F}:\mathcal{B}_\rho^X(0,c^2,b)\to\mathbb{Y}$ is Fr\'{e}chet differentiable with respect to $\kappa$.  

In our uniqueness proof we will use a slightly modified definition of the forward operator, namely 
\[
\widetilde{\mathbb{F}}=\widetilde{\mathcal{R}}_{c_0^2,\bconst}\circ \obsop\circ S =\widetilde{\mathbb{A}}_{c_0^2,\bconst}\circ S_* +\widetilde{\vec{a}}_{c_0^2,\bconst}
\]
where still $\obsop$, $S$, $S_*$ are defined as in \eqref{eqn:obs}, \eqref{S}, \eqref{Sstar} and, with two given time dependent functions $f_1$, $f_2$, 
\[
\begin{aligned}
\widetilde{\mathcal{R}}_{c_0^2,\bconst}:h\mapsto 
\Bigl(&(\widehat{f}_1(p_m^-)\textup{Res}(\widehat{h};p_m^+)+
\widehat{f}_1(p_m^+)\textup{Res}(\widehat{h};p_m^-))_{m\in\mathbb{N}},\\
&(\widehat{f}_2(p_m^-)\textup{Res}(\widehat{h};p_m^+)+
\widehat{f}_2(p_m^+)\textup{Res}(\widehat{h};p_m^-))_{m\in\mathbb{N}}\Bigr)
\end{aligned}
\]
\[
\begin{aligned}
\widetilde{\vec{a}}_{c_0^2,\bconst} = \Bigl(&(\tfrac{1}{\imath\sqrt{4c^2\lambda_m-\bconst^2}}
(\widehat{f}_1(p_m^-)\widehat{r}_{*m}(p_m^+)-\widehat{f}_1(p_m^+)\widehat{r}_{*m}(p_m^-))
B_{\lambda_m}[(1,\ldots,1)])_{m\in\mathbb{N}},\\
&(\tfrac{1}{\imath\sqrt{4c^2\lambda_m-\bconst^2}}
(\widehat{f}_2(p_m^-)\widehat{r}_{*m}(p_m^+)-\widehat{f}_2(p_m^+)\widehat{r}_{*m}(p_m^-))
B_{\lambda_m}[(1,\ldots,1)])_{m\in\mathbb{N}}\Bigr)
\end{aligned}
\]
\[
\begin{aligned}
&\widetilde{\mathbb{A}}_{c_0^2,\bconst}:(u_{*0},u_{*1})\mapsto \\
&\Bigl((\tfrac{\widehat{f}_1(p_m^-)-\widehat{f}_1(p_m^+)}{\imath\sqrt{4c^2\lambda_m-\bconst^2}} 
B_{\lambda_m}[(u_{*m1}^{\tau k}+\bconst u_{*m0}^{\tau k})_{k\in K^{\lambda_m}}]
+\tfrac{\widehat{f}_1(p_m^-)p_m^+-\widehat{f}_1(p_m^+)p_m^-}{\imath\sqrt{4c^2\lambda_m-\bconst^2}} B_{\lambda_m}[(u_{*m0}^{\tau k})_{k\in K^{\lambda_m}}])_{m\in\mathbb{N}},\\
& \quad+(\tfrac{\widehat{f}_2(p_m^-)-\widehat{f}_2(p_m^+)}{\imath\sqrt{4c^2\lambda_m-\bconst^2}} 
B_{\lambda_m}[(u_{*m1}^{\tau k}+\bconst u_{*m0}^{\tau k})_{k\in K^{\lambda_m}}]
+\tfrac{\widehat{f}_2(p_m^-)p_m^+-\widehat{f}_2(p_m^+)p_m^-}{\imath\sqrt{4c^2\lambda_m-\bconst^2}} B_{\lambda_m}[(u_{*m0}^{\tau k})_{k\in K^{\lambda_m}}])_{m\in\mathbb{N}}
\Bigr)
\end{aligned}
\]

Analogously to above, we can conclude the following differentiability result.
\begin{corollary}\label{cor:F}
Under the assumptions of Proposition~\ref{prop:S}, and with 
$(\widehat{f}_1(p_m^+)-\widehat{f}_1(p_m^-))_{m\in\mathbb{N}}$, 
$(\widehat{f}_2(p_m^+)-\widehat{f}_2(p_m^-))_{m\in\mathbb{N}}$ $\in\ell^\infty$, 
the operator $\widetilde{\mathbb{F}}$ is well-defined as a mapping $\mathcal{B}_\rho^X\to \mathbb{Y}$. 
If additionally $\cutoff\in W^{2,\infty}(0,\infty)$ then $\widetilde{\mathbb{F}}$ satisfies a Lipschitz estimate and can be extended to a (Lipschitz continuous) mapping $\widetilde{\mathbb{F}}:\mathcal{B}_{\tilde{\rho}}^{\tilde{X}}(0,c^2,b)\to \mathbb{Y}$, which is Fr\'{e}chet differentiable with Lipschitz continuous derivative.
\end{corollary}

\section{Uniqueness}\label{sec:uniqueness}
In this section we will prove  
\begin{enumerate}
\item[(i)] uniqueness of $\kappa$ (while $c_0$ does not necessarily need to be known for this purpose; the essential condition in which $c_0$ is implicitly involved is \eqref{eqn:ass_inj_Sigma_rem});
\item[(ii)] linearized uniqueness simultaneously of $(\kappa,c_0^2)$; 
\item[(iii)] linearized uniqueness simultaneously of $(\kappa,b_0)$; 
\end{enumerate}
from the boundary observations \eqref{eqn:obs}. 

A crucial step for this purpose will be to show that for every $(\kappa,c_0^2,\bconst) \in \mathcal{B}_\rho^X(0,c^2,b)$ with $\nabla\bconst=0$, the operator
\[
D\widetilde{\mathbb{F}}(0,c_0^2):=
\left(\begin{array}{c}\partial_\kappa\widetilde{\mathbb{F}}(0,c_0^2,\bconst)\\
d_{c_0^2}\widetilde{\mathbb{F}}(\kappa,c_0^2,\bconst)\end{array}\right)
=\left(\begin{array}{c}\widetilde{\mathcal{R}}_{c_0^2,\bconst}\circ \obsop\circ \partial_\kappa S(0,c_0^2,\bconst)\\
\widetilde{\mathcal{R}}_{c_0^2,\bconst}\circ \obsop\circ \partial_{c_0^2} S(0,c_0^2,\bconst)
\end{array}\right)
\]
is an isomorphism between $\mathbb{X}_1\times\mathbb{X}_2$ and $\mathbb{Y}$.
Here $\partial_\kappa\widetilde{\mathbb{F}}$ is the (partial) Fr\'{e}chet derivative according to the results from the previous section, while $d_{c_0^2}\widetilde{\mathbb{F}}$ is just a formal linearization in $c_0^2$ direction. 
On one hand, this implies that also $\partial_\kappa\widetilde{\mathbb{F}}(0,c_0^2,\bconst):\mathbb{X}_1\to\mathbb{Y}$ is a isomorphism and thus, by the Inverse Function Theorem, (i) follows.
On the other hand, we can directly conclude linearized uniqueness (ii) in the sense that  
$(\partial_\kappa\widetilde{\mathbb{F}}(0,c_0^2,\bconst),
d_{c_0^2}\widetilde{\mathbb{F}}(0,c_0^2,\bconst))$
and thus $(\obsop\circ \partial_\kappa S(0,c_0^2,\bconst),\obsop\circ \partial_{c_0^2} S(0,c_0^2,\bconst)$ is injective.
Similarly to the latter, one can also prove (iii).

\subsection{An isomorphism property of the linearized forward operator}\label{sec:isomorphism_kapp_c0}
Fixing $c_0$, $\bconst$, we denote by $(p_m^\pm)_{m\in\mathbb{N}}$ the sequence of poles according to \eqref{eqn:poles}.
Note that these are precisely the  poles of the Laplace transformed data $\widehat{h}$, provided $c_0$, $\bconst$ are the actual coefficients of the PDE for which this data has been taken, which we assume to hold in this section.

As in the linearized injectivity proof of \cite{nonlinearity_imaging_fracWest} we use an excitation $r$ that for $\kappa=0$  
leads to a space-time separable solution 
$u^0(x,t)=\phi(x)\psi(t)$ of \eqref{eqn:JMGT_init_D_cutoff}, 
namely 
\begin{equation}\label{eqn:r}
\begin{aligned}
&r(x,t):= \phi(x){\psi^\tau}''(t)+c^2(\calAc \phi)(x)\psi^\tau + \bconst \phi(x){\psi^\tau}'(t)\\
&\text{where }\psi^\tau=\tau\psi'+\psi
\end{aligned}
\end{equation}
with some given functions $\phi(x)$, $\psi(t)$ such that 
\begin{equation}\label{eqn:Lchi}
\begin{aligned}
&\phi\in\mathcal{D}(\calAc ), \quad \tfrac{1}{\phi}, \,\tfrac{1}{\Delta\phi}\in \dot{H}^s(\Omega)\text{ for some }s>\tfrac{d}{2}\\ 
&\psi(0)=\psi'(0)=\psi''(0)=0\\ 
&\nu_m:=\widehat{f}_1(p_m^-)\widehat{f}_2(p_m^+)-\widehat{f}_1(p_m^+)\widehat{f}_2(p_m^-)\not=0 \quad \mbox{ for all }m\in\mathbb{N}\\
&\text{ for }
f_1(t):=\tau\psi'(t)+\psi(t), \quad
f_2(t):=\cutoff(|\phi|_\Ltwo^2(\tau\psi'(t)+\psi(t))^2)(\psi^2)''(t).
\end{aligned}
\end{equation}
With this, the linearization of \eqref{eqn:JMGT_init_D_cutoff} at $\kappa=0$, $c_0^2$, $b_0=\bconst$ becomes
\begin{equation*}
\begin{aligned}
\underline{du}^\tau_{tt}+c^2\calAc  \underline{du}^\tau + \bconst \underline{du}^\tau_t &= 
f_2(t) \underline{d\kappa}(x)\phi^2(x) -f_1(t) \underline{dc}_0^2 \Delta \phi(x)
\quad \mbox{ in }\Omega\times(0,T)\\
\underline{du}(0)=0, \quad \underline{du}^\tau(0)=0, \quad \underline{du}^\tau_t(0)&=0 \quad \mbox{ in }\Omega\\
\underline{du}^\tau = \tau \underline{du}_t+\underline{du}.
\end{aligned}
\end{equation*}
Moreover, using separation of variables and the relaxation functions we can write 
\begin{equation}\label{Lapldu}
\Lapl{\underline{du}^\tau}(x,z)= \sum_{j=1}^\infty \tfrac{1}{\omega_{\lambda_j}(z)} \sum_{k\in K^{\lambda_j}}
\Bigl(\widehat{f}_2(z)\langle \underline{d\kappa}\phi^2, \varphi_j^k\rangle 
+\widehat{f}_1(z)\langle\underline{dc}^2_0(-\Delta\phi), \varphi_j^k\rangle\Bigr)\varphi_j^k(x). 
\end{equation}
Due to Lemma~\ref{lem:what}, taking the residues at the poles singles out the contributions of the individual eigenspaces to this sum, that is, after evaluation by the observation operator
\begin{equation}\label{eqn:idres}
\begin{aligned}
&\textup{Res}(\Lapl{\obsop\underline{du}};p_m^\pm)\\
&= \textup{Res}(\tfrac{1}{\omega_{\lambda_m}};p_m^\pm) 
\sum_{k\in K^{\lambda_m}} \Bigl(\Lapl{f_2}(p_m^\pm)\langle \underline{d\kappa}\phi^2, \varphi_m^k\rangle
+\Lapl{f_1}(p_m^\pm)\langle\underline{dc}^2_0(-\Delta\phi), \varphi_m^k\rangle
\Bigr) \obsop\varphi_m^k . 
\end{aligned}
\end{equation}

This allows us to exploit some cancellations and write the linearization of the forward operator at $\kappa=0$, $c_0^2=0$, $b_0=\bconst$ as
\begin{equation}\label{idforward}
\begin{aligned}
D\widetilde{\mathbb{F}}(0,c_0^2)(\dkappa,\underline{dc}_0^2)
=\Bigl(&(\tfrac{\nu_m}{\imath\sqrt{4c^2\lambda_m-\bconst^2}}
B_{\lambda_m}[(\langle \underline{d\kappa}\phi^2, \varphi_m^k\rangle)_{k\in K_{\lambda_m}}])_{m\in\mathbb{N}},\\
&(-\tfrac{\nu_m}{\imath\sqrt{4c^2\lambda_m-\bconst^2}}
B_{\lambda_m}[(\langle \underline{dc}^2_0(-\Delta\phi), \varphi_m^k\rangle)_{k\in K_{\lambda_m}}])_{m\in\mathbb{N}}
\Bigr)
\end{aligned}
\end{equation}
with $\nu_m$ as in \eqref{eqn:Lchi}.
As a consequence of Lemma~\ref{lem:resnon0} and \eqref{eqn:Lchi}, the factors $\frac{\nu_m}{\imath\sqrt{4c^2\lambda_m-\bconst^2}}$ appearing here do not vanish. 
Under condition 
\begin{equation}\label{condmu}
|\nu_m|^2\geq c_\nu \lambda_m^{s-1}, \quad m\in\mathbb{N} \text{ for some }s>\frac{d}{2}
\end{equation}
we have 
\begin{equation}\label{idforward1}
\begin{aligned}
&\|D\widetilde{\mathbb{F}}(0,c_0^2)(\dkappa,\underline{dc}_0^2)\|_{\mathbb{Y}}\\
&=\sum_{m\in\mathbb{N}}\tfrac{\lambda_m^2 |\nu_m|^2}{{4c^2\lambda_m-\bconst^2}}
\sum_{k\in K_{\lambda_m}} \Bigl(|\langle \underline{d\kappa}\phi^2, \varphi_m^k\rangle|^2
+|\langle \underline{dc}^2_0(-\Delta\phi),, \varphi_m^k\rangle|^2\Bigr)\\
&\geq \tfrac{c_\nu}{c^2}\Bigl(\|\underline{d\kappa}\phi^2\|_{\dot{H}^s(\Omega)}^2
+\|\underline{dc}^2_0(-\Delta\phi)\|_{\dot{H}^s(\Omega)}^2\Bigr)
\end{aligned}
\end{equation}
and thus, $D\widetilde{\mathbb{F}}(0,c_0^2)$ is an isomorphism between 
$\mathbb{X}_1\times\mathbb{X}_2$ and $\mathbb{Y}$, where
\[
\mathbb{X}_1:=\{\underline{d\kappa}\in L^\infty(\Omega)\, : \ \underline{d\kappa}\phi^2\in\dot{H}^s(\Omega)\}, \quad 
\mathbb{X}_2:\{\underline{dc}_0^2\in L^\infty(\Omega)\, : \ \underline{dc}^2_0(-\Delta\phi)\in\dot{H}^s(\Omega)\},
\]
so that $\mathbb{X}_1\times\mathbb{X}_2\times\{\bconst\}\subseteq X$.

\subsection{Linearized uniqueness of $(\kappa,c_0^2)$ for constant $b_0(x)\equiv\bconst$}\label{sec:uniqueness_kappac0}
From the isomorphism property of $D\widetilde{\mathbb{F}}(0,c_0^2)$, we conclude its injectivity and thus injectivity of $(\obsop\circ \partial_\kappa S(0,c_0^2,\bconst),\obsop\circ \partial_{c_0^2} S(0,c_0^2,\bconst)$.

\subsection{Linearized uniqueness of $(\kappa,b_0)$ for fixed $c_0(x)$}\label{sec:uniqueness_kappab0}
Similarly to the above isomorphism proof, also injectivity of $(\obsop\circ \partial_\kappa S(0,c_0^2,\bconst),\obsop\circ \partial_{b_0} S(0,c_0^2,\bconst)$ can be shown. 
To this end, we keep \eqref{eqn:r} and replace $f_1(t):=\tau\psi'(t)+\psi(t)$ in \eqref{eqn:Lchi} by $f_1(t):=\tau\psi''(t)+\psi'(t)$.
With this, the linearization of \eqref{eqn:JMGT_init_D_cutoff} at $\kappa=0$, $c_0^2$, $b_0=\bconst$ becomes
\begin{equation*}
\begin{aligned}
\underline{du}^\tau_{tt}+c^2\calAc  \underline{du}^\tau + \bconst \underline{du}^\tau_t &= 
f_2(t) \underline{d\kappa}(x)\phi^2(x) +f_1(t) \underline{db}_0 \phi(x)
\quad \mbox{ in }\Omega\times(0,T)\\
\underline{du}(0)=0, \quad \underline{du}^\tau(0)=0, \quad \underline{du}^\tau_t(0)&=0 \quad \mbox{ in }\Omega\\
\underline{du}^\tau = \tau \underline{du}_t+\underline{du}.
\end{aligned}
\end{equation*}
and so we just have to replace $\underline{dc}^2_0(-\Delta\phi)$ by $\underline{db}_0\phi$ in \eqref{Lapldu} and \eqref{eqn:idres}, \eqref{idforward}, \eqref{idforward1}.

\subsection{Local uniqueness of $\kappa$}\label{sec:uniqueness_kappa}
The fact that the operator $D\widetilde{\mathbb{F}}(0,c_0^2)$ is an isomorphism between 
$\mathbb{X}_1\times\mathbb{X}_2$ and $\mathbb{Y}$ also implies that $\partial_\kappa\widetilde{\mathbb{F}}(0,c_0^2,\bconst):\mathbb{X}_1\to\mathbb{Y}$ is a isomorphism.
This together with the differentiability result Corollary~\ref{cor:F} allows us to apply the  Inverse Function Theorem (see, e.g., \cite[Section 4.8]{Zeidler:1995})
and conclude that $\mathbb{F}(\cdot,c_0^2,\bconst):\mathbb{X}_1\to\mathbb{Y}$ is locally bijective (with a continuous inverse between these spaces) and thus, due to completeness of the eigensystem and our assumption that $\phi\not=0$ in $\Omega$ the full forward map $\kappa\to\obsop u$ is locally injective.
More precisely, we even have local well-posedness of the inverse problem in the $\mathbb{X}_1$-$\mathbb{Y}$ topologies.

\begin{theorem}\label{thm:uniqueness_kappa}
Assume that the conditions of Proposition~\ref{prop:S} are satisfied, that \eqref{eqn:ass_inj_Sigma_rem} holds, $r$ takes the form \eqref{eqn:r} with 
$\phi$, $\psi$ satisfying \eqref{eqn:Lchi}, \eqref{condmu}, and that $c_0(x)$, $b_0(x)\equiv \bconst$ are the actual sound speed and damping coefficients. Then there exists $\rho>0$ such that for any $\kappa ,\tilde{\kappa}\in B_\rho^{\mathbb{X}_1}(0)$, 
equality of the observations $\obsop S(\kappa,c_0^2,\bconst)(t)=\obsop S(\tilde{\kappa},c_0^2,\bconst)(t)$, $t>0$ implies $\kappa=\tilde{\kappa}$.
\end{theorem}

This does not require $c_0(x)$ nor $\calAc $ to be known, just the conditions \eqref{eqn:r}, \eqref{eqn:Lchi}, \eqref{eqn:ass_inj_Sigma_rem} to be satisfied by the (possibly unknown) operator $\calAc $ and its eigensystem.

Existence of an excitation $r$ according to \eqref{eqn:r} such that $e^{\alpha/2\cdot} r\in L^2(\Ltwo)$ and \eqref{eqn:Lchi} holds might be hard to reconcile with  \eqref{condmu}, in particular in dimensions larger than one, where \eqref{condmu} requires growth of $\nu_n$. However, we can use the result to recover an arbitrary number $M$ of generalized Fourier components of $\kappa\phi^2$, assuming only \eqref{eqn:r}, $e^{\alpha/2\cdot} r\in L^2(\Ltwo)$ and \eqref{eqn:Lchi} to hold, by setting $c_\nu:=\max\{|\nu_m|^{-2}\lambda_m^{s-1}\, : \, m\in\{1,\ldots,M\}\}>0$ for a given $s>\frac{d}{2}$ and working in the space
$\mathbb{X}_{1M}:=\{\underline{d\kappa}\in L^\infty(\Omega)\, : \ \underline{d\kappa}\phi^2\in\text{span}(\varphi_j^k)_{j\in\{1,\ldots, M\},\, k\in K^{\lambda_j}}\}\subseteq \dot{H}^s(\Omega)$.

\begin{corollary}\label{cor:uniqueness_kappa_M}
Assume that the conditions of Proposition~\ref{prop:S} are satisfied, that \eqref{eqn:ass_inj_Sigma_rem} holds, $r$ takes the form \eqref{eqn:r} with 
$\phi$, $\psi$ satisfying \eqref{eqn:Lchi}, 
 and that $c_0(x)$, $b_0(x)\equiv \bconst$ are the actual sound speed and damping coefficients. Then there exists $\rho>0$ such that for any $\kappa ,\tilde{\kappa}\in B_\rho^{\mathbb{X}_{1M}}(0)$, equality of the observations $\obsop S(\kappa,c_0^2,\bconst)(t)=\obsop S(\tilde{\kappa},c_0^2,\bconst)(t)$, $t>0$ implies $\kappa=\tilde{\kappa}$.
\end{corollary}

\section*{Acknowledgment}
This work was supported by the Austrian Science Fund FWF  http://dx.doi.org/10.13039/501100002428 under the grant P36318.

\end{document}